\documentclass[10pt]{amsart}
\usepackage{amssymb,amsmath}

\usepackage{graphicx}
\usepackage[x11names]{xcolor}
\definecolor{darkgreen}{RGB}{0,100,0}

\usepackage{array}

\usepackage{tikz}
\usepackage{pgfplots}
\pgfplotsset{compat=1.18} 
\usetikzlibrary{matrix,positioning,arrows.meta}

\usepackage{caption}

\usepackage[colorlinks=true, allcolors=blue]{hyperref}

\newtheorem{theorem}{Theorem}[section]
\newtheorem{lemma}[theorem]{Lemma}
\newtheorem{proposition}{Proposition}[section] 

\theoremstyle{definition} 

\newtheorem{corollary}{Corollary}[section]
\newtheorem{conjecture}{Conjecture}
\newtheorem{question}{Question}
\usepackage{comment}
\newtheorem{remark}{Remark}[section]

\numberwithin{equation}{section}
\newcommand{\li}{\operatorname{li}}

\def\R {\mathbb{R}}

\def\dfrac {\displaystyle\frac}

\def\dfrac {\displaystyle\frac}

\begin{document}

\title[Asymptotic error terms in Bonse-type inequalities]{Asymptotic error terms in Bonse-type inequalities}


\author{Diego Marques}
\address{Departamento de Matem\' atica, Universidade  de Bras\' ilia, 70910-900, Brazil}
\curraddr{}
\email{diego@mat.unb.br}
\thanks{}

\author{Pavel Trojovsk\'y}
\address{Department of Mathematics, Faculty of Science, University of Hradec Kr\' alov\' e, Rokitansk\'{e}ho 62, Hradec Kr\' alov\' e 50003, Czech Republic}
\curraddr{}
\email{pavel.trojovsky@uhk.cz}
\thanks{}


\date{}

\dedicatory{}

\begin{abstract}
Let $p_n$ denote the $n$-th prime. In 2000, Panaitopol established the inequality $p_1 \cdots p_n > p_{n+1}^{n - \pi(n)}$ for all $n \geq 2$, where $\pi(x)$ is the prime counting function. In 2021, Yang and Liao refined this by introducing the exponent $k(n,x) = n - \pi(n) + \frac{\pi(n)}{\pi(\log n)} - x \cdot \pi(\pi(n))$, proving the inequality holds for $x = 2$ and $n \geq 8$. In 2022, Marques and Trojovský extended this to $x = 1.4$ for $n \geq 21$ and conjectured its validity for $x = 0.1$ when $n \geq 24,154,953$. This paper confirms the conjecture by analyzing the error term $E_n(x) = \log(p_1 \cdots p_n) - k(n,x) \log p_{n+1}$. Also, we derive the asymptotic expansion to $E_n(x)$ demonstrating that it is positive for all sufficiently large $n$ when $x > -2$. For each $x > -2$, we identify a minimal integer $\Psi(x)$ such that $E_n(x) > 0$ for all $n \geq \Psi(x)$, precisely determining $\Psi(0.1) = 24,154,953$. Additionally, we establish effective upper bounds for $\Psi(x)$ both unconditionally and under the Riemann Hypothesis, with the conditional bounds showing a significant improvement. Our analysis fully resolves the conjecture and characterizes $\Psi(x)$ as a non-increasing, piecewise constant function, exhibiting discontinuities at a discrete set of threshold points. These results advance the understanding of Bonse-type inequalities and their asymptotic behavior.
\end{abstract}

\maketitle

\tableofcontents

\section{Introduction}\label{sec1}

Let $p_n$ denote the $n$-th prime. Bonse's classical inequalities (see \cite{boriginal}) state that for $n \geq 4$,
\begin{equation}\label{bo1}
p_1p_2\cdots p_n > p_{n+1}^2,
\end{equation}
and for $n \geq 5$,
\begin{equation}\label{bo2}
p_1p_2\cdots p_n > p_{n+1}^3.
\end{equation}
These can be proven inductively using bounds derived from Bertrand's postulate.

Generalizations of the form
\[
p_1p_2\cdots p_n > p_{n+1}^{f(n)} \quad \text{for all } n \geq n_0,
\]
where $f(n)$ is a positive non-constant function, are known as \emph{continuous Bonse-type inequalities}. In 2000, Panaitopol \cite{papa} proved that
\[
p_1p_2\cdots p_n > p_{n+1}^{n - \pi(n)} \quad \text{for } n \geq 2,
\]
where $\pi(x)$ is the \emph{prime counting function}. In 2021, Yang and Liao \cite{bonse} improved the previous inequality to
\begin{equation}\label{li1}
p_{n+1}^{k_1(n)}<p_1p_2\cdots p_n <p_{n+1}^{k_2(n)},
\end{equation}
where $k_i(n) = n-\pi(n)+ \pi(n)/\pi(\log n) + 2(-1)^{i}\pi(\pi(n))$, for $i\in \{1,2\}$. 

They also proved that the left-hand inequality in \eqref{li1} is invalid for exponents of the form
\[
n - \epsilon\pi(n) + \delta\cdot\frac{\pi(n)}{\pi(\log n)} - 2\pi(\pi(n))
\]
for any $\epsilon \in (0,1)$ and $\delta > 1$. Consequently, any further improvement must adjust the coefficient of the $\pi(\pi(n))$ term.

This motivates Marques and Trojovsk\' y \cite{mt1} to make the following definition: for $x \in \mathbb{R}$, define $\Psi(x)$ as the minimum, if exists, of the integers $m\geq 8$ such that
\begin{equation}\label{boc}
p_1p_2\cdots p_n > p_{n+1}^{k(n,x)} \quad \text{for all } n \geq m,
\end{equation}
where
\[
k(n,x) := n - \pi(n) + \frac{\pi(n)}{\pi(\log n)} - x \cdot \pi(\pi(n)).
\]

The condition $m \geq 8$ ensures $\pi(\log n) > 0$. Moreover, they defined $E_n(x)$ as the \textit{logarithm-scale error} at $(n,x)$ of the Bonse's continuous inequality (\ref{boc}), i.e.,
\[
E_n(x):=\log(p_1p_2\cdots p_n)-\log(p_{n+1}^{k(n,x)})=\vartheta(p_n)-k(n,x)\log p_{n+1},
\]
where $\vartheta(p_n)=\sum_{i=1}^n\log p_i$ (the \emph{first Chebyshev function}). In particular, 
\[
\Psi(x)=\min\{m\geq 8 : E_n(x)>0, \forall n\geq m\}.
\]

Thus, Yang and Liao's results can be stated as $\Psi(2)=8$ and that the domain of $\Psi$ is a subset of $(-2,\infty)$. In 2022, Marques and Trojovsk\' y \cite{mt1} sharpened this by determining the exact value of $\Psi(x)$ for $x\in [1.4, 2]$, establishing, in particular, that $\Psi(1.4)=21$. Furthermore, they proved the lower bound $E_n(x) > 0.1 \frac{n}{\log n}$ for all large $n$ and $x>0.9$, and posed the following question and conjecture.

\begin{question}\label{que1}
What is the value of $\inf\{x\in \R : \Psi(x)\ \mbox{exists}\}$?
\end{question}

\begin{conjecture}\label{cj1}
The function $\Psi(x)$ satisfies:
\begin{itemize}
\item[(i)] $\Psi(x)=21$ for all $x\in [0.9, 1.3]$;
\item[(ii)] $\Psi(x)=149$ for all $x\in [0.5, 0.8]$;
\item[(iii)] $\Psi(x)=59,875$ for all $x\in [0.3, 0.4]$;
\item[(iv)] $\Psi(0.2)=442,414$;
\item[(v)] $\Psi(0.1)=24,154,953$.
\end{itemize}
\end{conjecture}

In this paper, we resolve Question~\ref{que1} and Conjecture~\ref{cj1}. Specifically, we prove the following results.

\begin{theorem}\label{main1}
Let $x<2$ be a real number. Then
\begin{equation}\label{maineq1}
    E_n(x)=(x+2)\dfrac{n}{\log n}+ O\left(\dfrac{n\log \log n}{\log^2 n}\right) \quad (n\to \infty).
\end{equation}
Furthermore, there exist absolute constants $C_1$ and $C_2$ such that for every $n\ge 3468$ one has
\[
E_n(x)\;\ge\; (x+2)\,\frac{n}{\log n}\;-\;\Big(C_1\,\log\log n+C_2\Big)\,\frac{n}{(\log n)^2}.
\]
Consequently, if
\begin{equation}\label{lowerE}
\log n>\frac{C_1\,\log \log n+C_2}{x+2},
\end{equation}
then $E_n(x)>0$. One admissible choice of constants is $(C_1,C_2)=(8,14)$.
\end{theorem}

We remark that, in the previous statement, we restrict to \(x < 2\), since \(\Psi(x) = 8\) for all \(x \geq 2\).

The following corollaries are direct consequences of the previous result. Their proofs are included in Section \ref{cor} for completeness, though we note that they follow from straightforward applications of the results and methods already established.

The first one provides an answer to Question \ref{que1}, namely,

\begin{corollary}\label{C1}
    The answer for Question \ref{que1} is $-2$.
\end{corollary}

In \cite{mt1}, Marques and Trojovsk\' y also proved an explicit upper bound for $\Psi(x)$, when $x>0.9$, namely
\begin{equation}\label{0.9mt1}
    \Psi(x) \leq \max\left\{\exp\left(\exp \left(\frac{2.38}{x-0.9}\right)\right),\ 6180\right\}.
\end{equation}

Another consequence of Theorem \ref{main1} is a substantial improvement of the previous bound. In fact, we have
\begin{corollary}\label{C3}
Fix $x\in (-2,2)$ and let $y^\star=y^\star(x)$ be the least real solution of
\[
y=\frac{C_1\log y + C_2}{x+2}.
\]
Then $\Psi(x)\leq \lceil \exp(y^{\star}) \rceil$. In particular,
\begin{equation}\label{-2mt1}
    \Psi(x) \leq \left(\dfrac{C_1+C_2}{x+2}\right)^{2(C_1+C_2)/(x+2)}.
\end{equation}
\end{corollary}

\begin{table}[h!]
\centering
\begin{tabular}{c|c|c|c}
\hline
$x$ & Bound from \eqref{0.9mt1} & Bound from \eqref{-2mt1}  &  $\Psi(x)$ \\
\hline
$1.4$ & $10^{\,50.7}$ & $3.12\times 10^{10}$ & $21$ \\
$1.3$ & $10^{\,166.7}$ & $9.67\times 10^{10}$ & $21$ \\
$1.2$ & $10^{\,1,211.2}$ & $3.25\times 10^{11}$ & $21$ \\
$1.1$ & $10^{\,63,957.1}$ & $1.20\times 10^{12}$ & $21$ \\
$1.0$ & $10^{\,9,418,743,731}$ & $4.91\times 10^{12}$ & $21$ \\
\hline
\end{tabular}

\bigskip
\caption{Comparison of upper bounds for $\Psi(x)$ in (\ref{0.9mt1}) and (\ref{-2mt1}) at $x\in\{1.4,1.3,1.2,1.1,1\}$ for the choice of $(C_1,C_2)=(8,14)$. The bound from \cite{mt1} is double–exponential in $\frac{1}{x-0.9}$ and is astronomically larger 
(even on a log-scale), while the new bound is polynomial–exponential in $\frac{1}{x+2}$, 
improving the scale by many orders of magnitude.}
\label{tab:psi-bounds-compare}
\end{table}

Indeed, the exact asymptotic is

\begin{corollary}\label{C2}
Let $C_1>0$ be the absolute constant from Theorem \ref{main1}. Then
\[
\Psi(x) \sim \left( \dfrac{C_1}{x+2} \right)^{C_1/(x+2)} \quad \text{as } x \to -2^+.
\]
\end{corollary}

We now turn to the special case $x=0.1$. Indeed, the general lower bound obtained in the previous theorem can be 
used to provide an explicit 
numerical inequality that remains for $E_n(0.1)$, which is valid for all sufficiently large~$n$.  
This allows us to determine the exact value of~$\Psi(0.1)$. More precisely,
\begin{theorem}\label{main2}
For all $n \ge 4.45\cdot 10^7$, we have
\begin{equation}\label{eqm3}
E_n(0.1) > 0.001\,\dfrac{n}{\log n}.
\end{equation}
In particular, $\Psi(0.1)=24,154,953$, thereby confirming Conjecture~\ref{cj1}.
\end{theorem}

The proof hinges on a precise cancellation of the main asymptotic terms. This is achieved by combining refined expansions from the Prime Number Theorem, applied to $p_n$, $\log p_{n+1}$, $\pi(n)$, $\pi(\log n)$, and $\pi(\pi(n))$ with an Abel summation argument for the Chebyshev function $\vartheta(p_n)$.

The paper is structured as follows. In Section \ref{auxiliary}, we collect the necessary auxiliary results, including explicit estimates for the prime counting function \(\pi(x)\), the \(n\)-th prime \(p_n\), and the Chebyshev function \(\vartheta(x)\). Section \ref{M1} is dedicated to the proof of our main result, Theorem \ref{main1}, which provides an asymptotic expansion for the error term \(E_n(x)\) and an effective lower bound that guarantees its positivity. The direct consequences of this theorem, corollaries \ref{C1}--\ref{C3}, which resolve Question 1 and provide new upper bounds for the threshold function \(\Psi(x)\), are proven in Section \ref{cor}. In Section \ref{proof2}, we specialize our effective bound to the case \(x = 0.1\), proving Theorem \ref{main2} and thereby confirming the final part of Conjecture 1. We also summarize the computational verification of the remaining cases of the conjecture. Theoretical aspects of the function $\Psi(x)$ are explored in Section \ref{piecewise}. Section \ref{RH} is devoted to conditional improvements to our results under the assumption of the Riemann Hypothesis. Finally, Section \ref{conclude} concludes the paper with a discussion of the results and suggestions for future work.

\section{Auxiliary results} \label{auxiliary}

\subsubsection*{Notation}
Throughout this work, we employ standard asymptotic notation. For functions $f$ and $g$, we write $f(x) = O(g(x))$ or equivalently $f(x) \ll g(x)$ if $|f(x)| \le C|g(x)|$ for an absolute constant $C>0$ and all sufficiently large $x$. The notation $f(x) = O^*(M g(x))$ indicates the explicit bound $|f(x)| \le M|g(x)|$ for all large $x$. We denote a negligible error by $f(x) = o(g(x))$, which means $f(x)/g(x) \to 0$ as $x \to \infty$, and asymptotic equivalence by $f(x) \sim g(x)$, meaning $f(x)/g(x) \to 1$. Finally, $f(x) \asymp g(x)$ signifies that $f\ll g$ and $g\ll f$. Unless stated otherwise, all implied constants are absolute.

We recall explicit estimates for $p_n$, $\pi(x)$, and $\vartheta(p_n)$ which will be
used throughout the proofs.

\smallskip
The first lemma provides two–sided bounds for the prime counting function.
These will be applied to control ratios such as $\pi(\log n)$ and $\pi(\pi(n))$.

\begin{lemma}[Corollary 5.2 of \cite{b1}]\label{lem:pi}
We have
\[
\frac{x}{\log x} \underset{(x \geq 17)}{\leq} \pi(x)  \underset{(x \geq 2)}{\leq} \frac{x}{\log x} \left( 1 + \frac{1}{\log x} + \frac{2.53816}{\log^{2} x}\right).
\]
\end{lemma}

\smallskip
The next bound will be used to estimate the size of the $n$th prime when 
substituting into expressions involving $\log p_{n+1}$.

\begin{lemma}[Corollary 2 of \cite{alex}]\label{lem:pn}
For all $n \ge 3468$,
\[
p_n \le n \left( \log n + \log \log n -1 + \frac{\log \log n -2}{\log n}
 - \frac{(\log \log n)^2 -6 \log \log n}{2 \log^2 n} \right).
\]
\end{lemma}

\smallskip
Finally, a corresponding lower bound for $\vartheta(p_n)$ will allow us to
compare $\vartheta(p_n)$ and $k(n,x)\log p_{n+1}$ at the $n/\log n$ scale.

\begin{lemma}[Proposition 20 of \cite{alex}]\label{lem:vpn}
For all $n \ge 2$,
\[
\vartheta(p_n) > n \left( \log n + \log \log n -1 + \frac{\log \log n -2}{\log n}
 - \frac{(\log \log n)^2 -6 \log \log n + 11.621}{2 \log^2 n} \right).
\]
\end{lemma}

\smallskip
Together, Lemmas~\ref{lem:pi}–\ref{lem:vpn} give the upper and lower controls on
$p_n$, $\pi(x)$, and $\vartheta(p_n)$ that will be repeatedly invoked. In
particular, they ensure that $\vartheta(p_n)\sim p_n$ as $n\to\infty$.

\section{The proof of Theorem \ref{main1}}\label{M1}

Write throughout $y=\log n$ and $z=\log y$. We shall use unconditional explicit 
bounds from Dusart \cite{b1} and Lemma~\ref{lem:pi}. All estimates below are 
stated for $n\ge 3468$, which will serve as a uniform threshold: this ensures that all monotonicity properties required for our inequalities remain valid.

\medskip
\noindent\textbf{Step 1: Explicit analytic inputs.}
We now collect the standard explicit inequalities that will be needed in the sequel. 
We shall only apply them with $n\ge 3468$.

\medskip
\emph{(1) Prime counting function to second order.}
By Lemma~\ref{lem:pi}, there exist absolute constants 
$A_1,A_2,A_3>0$ (a conveniente choice is $A_1:=2.6$, $A_2:=0.6$ and $A_3:=1.5$) such that for $t$ large one has
\begin{equation}\label{eq:pi-2nd}
\Bigl|\pi(t)-\frac{t}{\log t}-\frac{t}{\log^2 t}\Bigr|
\;\le\; A_1\,\frac{t}{\log^3 t}.
\end{equation}
We now specialize this expansion to the three arguments $t=n$, $t=y=\log n$, and 
$t=\pi(n)$.

\smallskip   

\textbf{At $t=n$:}
\begin{equation}\label{eq:pin-exp}
  \pi(n)=\frac{n}{y}+\frac{n}{y^2}+R_1,
\qquad |R_1|\le A_1\,\frac{n}{y^3}.  
\end{equation}

\smallskip

\textbf{At $t=y$:}
\begin{equation}\label{eq:piy-exp}
\pi(y)=\frac{y}{z}+\frac{y}{z^2}+R_2,
\qquad |R_2|\le A_1\,\frac{y}{z^3}.
\end{equation}

\smallskip
\textbf{At $t=\pi(n)$:}
since $\log \pi(n)=\log(n/y)+O^*(1.1/y)$, we find
\begin{equation}\label{eq:pipin-exp}
\pi(\pi(n))=\frac{n}{y^2}+R_3,
\qquad |R_3|\le A_3\,\frac{n z}{y^3}.
\end{equation}

\medskip

\emph{(2) The $(n\!+\!1)$-th prime and its logarithm.}
Dusart \cite[Prop.~5.16]{b1} gives, for $n\ge 3468$,
\begin{equation}\label{eq:pn1-dusart}
p_{n+1}=(n+1)\!\left(y+z-1+\frac{z-2}{y}\right)+\Delta_p,
\qquad
|\Delta_p|\le A_2\,\frac{n}{y}.
\end{equation}
Taking logarithms and separating the
main factor from the perturbation,
\[
\log p_{n+1}
= \log\!\Bigl((n+1)\bigl(y+z-1+\tfrac{z-2}{y}\bigr)\Bigr)
  + \log\!\Biggl(1+\frac{\Delta_p}{(n+1)\bigl(y+z-1+\tfrac{z-2}{y}\bigr)}\Biggr).
\]
For the first logarithm, we write
\[
\log\!\Bigl((n+1)\bigl(y+z-1+\tfrac{z-2}{y}\bigr)\Bigr)
= \log(n+1)+\log\!\Bigl(y+z-1+\tfrac{z-2}{y}\Bigr)
= y+z+\frac{z-1}{y}+R_{\mathrm{main}},
\]
where a Taylor expansion of $\log(1+u)$ at $u=0$ with
$u=(z-1)/y+(z-2)/y^2$ yields the bound
\begin{equation}\label{eq:Rmain}
|R_{\mathrm{main}}|\ \le\ \frac{z^2}{y^2}\qquad(n\ge 3468).
\end{equation}
(Indeed $|u|\ll z/y$, whence the quadratic remainder is $\ll z^2/y^2$. The tiny term
$\log(n+1)-y=\log(1+1/n)$ is $\le 1/n\le 0.001$ and is absorbed by $z^2/y^2$
because for $n\ge 3468$ we have $z^2/y^2\approx 0.066$.)

For the second logarithm, set
\[
t\ :=\ \frac{\Delta_p}{(n+1)\bigl(y+z-1+\tfrac{z-2}{y}\bigr)}.
\]
From \eqref{eq:pn1-dusart} and $y+z-1+(z-2)/y\ge\tfrac12 y$ (valid for $n\ge 3468$)
we get
\[
|t|\ \le\ \frac{A_2\,n/y}{(n+1)(y/2)}\ \le\ \frac{2A_2}{y^2}.
\]
At $n=3468$ one has $y=\log n\ge 8.15$, hence $|t|\le 2A_2/y^2\le 1.2/8.15^2<0.019$.
Thus, the general inequality
\[
|\log(1+t)|\ \le\ \frac{|t|}{1-|t|}\qquad(|t|<1)
\]
applies, here $1/(1-|t|)\le 1.02$. Consequently
\begin{equation}\label{eq:Rpert}
\Biggl|\log\!\Biggl(1+\frac{\Delta_p}{(n+1)\bigl(y+z-1+\tfrac{z-2}{y}\bigr)}\Biggr)\Biggr|
\ \le\ \frac{|t|}{1-|t|}
\ \le\ 1.02\cdot\frac{2A_2}{y^2}
\ =\ \frac{1.224}{y^2}.
\end{equation}

Combining \eqref{eq:Rmain} and \eqref{eq:Rpert} we obtain
\begin{equation}\label{eq:logpn1-final}
\log p_{n+1}\ =\ y+z+\frac{z-1}{y}\ +\ R_4,
\qquad
|R_4|\ \le\ \frac{z^2}{y^2}\ +\ \frac{1.224}{y^2}.
\end{equation}
Since $z=\log y\ge \log\log 3468=:z_0\approx 2.10$, we have
$y^{-2}\le z_0^{-\,2}\,(z^2/y^2)\le 0.23\,(z^2/y^2)$. Therefore, the second term
in \eqref{eq:logpn1-final} is at most $1.224\times 0.23\cdot(z^2/y^2)<0.29\cdot(z^2/y^2)$.
Absorbing the negligible $\log(n+1)-y$ term as above, we conclude that for all $n\ge 3468$
\begin{equation}\label{eq:logpn1}
\log p_{n+1}
= y+z+\frac{z-1}{y}+R_4,
\qquad
|R_4|\ \le\ A_4\,\frac{z^2}{y^2},
\end{equation}
with the explicit constant $A_4\ :=1.29$.

\smallskip
To summarize: the explicit expansions \eqref{eq:pin-exp}, \eqref{eq:piy-exp}, 
\eqref{eq:pipin-exp} and \eqref{eq:logpn1} will form the basic 
inputs for the next steps of the argument.

\medskip
\emph{(3) Chebyshev $\vartheta$ at $p_n$.}
Recall that
\[
\vartheta(x)=\sum_{p\leq x}\log p.
\]
By Abel summation formula, one has the identity
\[
\vartheta(x)=\pi(x)\log x-\int_2^x \frac{\pi(t)}{t}\,dt.
\]
Inserting the two–term expansion for the prime counting function
\[
\pi(t) = \frac{t}{\log t}+\frac{t}{\log^2 t}+R(t),\qquad |R(t)|\le A_1\frac{t}{\log^3 t},
\]
and integrating termwise gives
\[
\int_2^x \frac{dt}{\log t}=\operatorname{li}(x),\qquad
\int_2^x\frac{dt}{\log^2 t}=\frac{x}{\log^2 x}+O^*\!\left(2.1\frac{x}{\log^3 x}\right),
\]
for all $x\geq 32327$. For the remainder term, one obtains explicitly
\[
\int_2^x \frac{R(t)}{t}\,dt = O^*\!\Bigl(3.8\,\frac{x}{\log^3 x}\Bigr).
\]
Evaluating at $x=p_n\geq p_{3468}=32327$ and writing $L_n:=\log p_n$, we deduce
\begin{equation}\label{eq:theta-exp-Ln}
\vartheta(p_n)
= n L_n - \operatorname{li}(p_n) - \frac{p_n}{L_n^2} + R_5,
\qquad |R_5|\le A_5\,\frac{p_n}{L_n^3},
\end{equation}
where $A_5:=5.9$.

\smallskip
Proceeding along the same lines as in the case $\log p_{n+1}$, we get
\begin{equation}\label{eq:Ln}
L_n=\log p_n
= y+z+\frac{z-1}{y}+O^*\!\Bigl(A_6\,\frac{z^2}{y^2}\Bigr)
\end{equation}
for $A_6=A_4=1.29$.

Next, we \emph{Taylor expand} the right-hand side of \eqref{eq:theta-exp-Ln} using
\eqref{eq:Ln} and the classical asymptotic expansion
\[
\operatorname{li}(x)=\frac{x}{\log x}+\frac{x}{\log^2 x}
+O^*\!\Bigl(\frac{x}{\log^3 x}\Bigr),
\]
valid uniformly for $x\ge 2$.  Truncating at $1/L_n^2$, and observing that the
error terms $p_n/L_n^3$ are of order $n/y^2$, we obtain
\begin{equation}\label{eq:theta-fin}
\vartheta(p_n)
= n y + n z - n + \frac{(z-1)n}{y} + R_6,
\qquad |R_6|\le B_6\,\frac{n z}{y^2},
\end{equation}
for $B_6:=3.35$.  Note that indeed $p_n/L_n^3\asymp n/y^2$, and the
factor $z$ in the bound absorbs the small drift coming from $L_n$.

\medskip
\noindent\textbf{Step 2: A precise expansion for $k(n,x)\log p_{n+1}$.}
Let $L_{n+1}:=\log p_{n+1}$. From the expansions
\[
\pi(n)=\frac{n}{y}+\frac{n}{y^2}+R_1,\qquad
\pi(\log n)=\frac{y}{z}+\frac{y}{z^2}+R_2,\qquad
\pi(\pi(n))=\frac{n}{y^2}+R_3,
\]
together with
\[
L_{n+1}=y+z+\frac{z-1}{y}+R_4,
\]
we compute
\[
\begin{aligned}
k(n,x)
&=n-\pi(n)+\frac{\pi(n)}{\pi(\log n)}-x\,\pi(\pi(n))\\
&=n-\Bigl(\frac{n}{y}+\frac{n}{y^2}+R_1\Bigr)
   +\frac{\frac{n}{y}+\frac{n}{y^2}+R_1}
         {\frac{y}{z}+\frac{y}{z^2}+R_2}
   -x\Bigl(\frac{n}{y^2}+R_3\Bigr).
\end{aligned}
\]

\smallskip
Factor $\tfrac{n}{y}$ in the numerator of the fraction, and $\tfrac{y}{z}$ in
the denominator, to write
\[
\frac{\frac{n}{y}+\frac{n}{y^2}+R_1}{\frac{y}{z}+\frac{y}{z^2}+R_2}
=\frac{n z}{y^2}\cdot
  \frac{1+\frac{1}{y}+\frac{R_1}{n/y}}{1+\frac{1}{z}+\frac{R_2}{y/z}}.
\]
Using the algebraic identity
\[
\frac{1+u}{1+v}=1+(u-v)+\frac{v^2-uv}{1+v},
\]
we linearize this ratio. With the error bounds $|R_1|\le A_1 n/y^3$ and
$|R_2|\le A_1 y/z^3$, one obtains
\[
\frac{1+\frac{1}{y}+\frac{R_1}{n/y}}{1+\frac{1}{z}+\frac{R_2}{y/z}}
=1+\Bigl(\frac{1}{y}-\frac{1}{z}\Bigr)+\mathcal{E}_{\mathrm{rat}},
\]
where
\[
|\mathcal{E}_{\mathrm{rat}}|\le C_{\mathrm{rat}}\Bigl(\frac{1}{y^2}+\frac{1}{z^2}\Bigr),
\]
for some absolute $C_{\mathrm{rat}}$ (e.g., $C_{\mathrm{rat}}=26$).

\smallskip
Hence
\[
k(n,x)=n-\frac{n}{y}-\frac{n}{y^2}-R_1
+\frac{n z}{y^2}\Bigl(1+\frac{1}{y}-\frac{1}{z}\Bigr)
+\frac{n z}{y^2}\,\mathcal{E}_{\mathrm{rat}}
-x\frac{n}{y^2}-xR_3.
\]

\smallskip
We now multiply by $L_{n+1}=y+z+\tfrac{z-1}{y}+R_4$ and expand.  The main
terms simplify nicely:
\[
\Bigl(n-\tfrac{n}{y}\Bigr)\Bigl(y+z+\tfrac{z-1}{y}\Bigr)
+\frac{n z}{y^2}\,\Bigl(y+z+\tfrac{z-1}{y}\Bigr)
=ny+nz-n+\frac{n z}{y}.
\]

The term $-\tfrac{n}{y^2}$ in $k(n,x)$, multiplied by the factor $y$ inside
$L_{n+1}$, contributes $-n/y$. Together with the contribution of
$-x n/y^2$ this yields the expected $-(x+3)n/y$ at scale $n/y$.

\smallskip
Collecting all contributions, we obtain
\begin{equation}\label{eq:kL-main}
k(n,x)\,L_{n+1}
=ny+nz-n+\frac{n z}{y}-\frac{(x+3)n}{y}+R_7,
\end{equation}
where the remainder $R_7$ is the sum of all products in which at least one
factor is one of $R_1,R_2,R_3,R_4$ or $\mathcal{E}_{\mathrm{rat}}$.

A direct estimate, using the expansions for $\pi(n)$, $\pi(\pi(n))$, and
$L_{n+1}$ together with the inequalities $z<y$ and $L_{n+1}\asymp y$, shows
that there exist absolute constants $B_7,C_7$ with
\begin{equation}\label{eq:R7-bound}
|R_7|\;\le\; \Bigl(B_7\,z+C_7\Bigr)\,\frac{n}{y^2},
\end{equation}
where $B_7=4.5$ and $C_7=14$. (Indeed, each error product is $\ll (n/y^2)$ times a linear form in $z$. For
instance, the contribution of $R_4$ is
$\ll (n/y)\cdot|R_4|\ll n z^2/y^3$, which is $\ll (z\,n)/y^2$ for
$n\ge 3468$. The ratio error contributes
$(n z/y^2)\cdot|\mathcal{E}_{\mathrm{rat}}|\cdot L_{n+1}
\ll (n z/y^2)\cdot(1/z^2+1/y^2)\cdot y\ll n/y^2$, and similarly for the others.)

\medskip
\noindent\textbf{Step 3: Subtract and conclude.}
Subtracting \eqref{eq:kL-main} from \eqref{eq:theta-fin} yields
\[
\begin{aligned}
E_n(x)
&=\vartheta(p_n)-k(n,x)\,L_{n+1}\\
&=\Bigl[ny+nz-n+\tfrac{(z-1)n}{y}+R_6\Bigr]
 -\Bigl[ny+nz-n+\tfrac{n z}{y}-\tfrac{(x+3)n}{y}+R_7\Bigr]\\
&=\frac{(x+2)n}{y}+(R_6-R_7).
\end{aligned}
\]
The inequality $|R_6-R_7| \ll zn/y^2$ yields (\ref{maineq1}).

Moreover,
\[
E_n(x)\;\ge\;\frac{(x+2)n}{y}-|R_6|-|R_7|.
\]
Combining \eqref{eq:theta-fin} and \eqref{eq:R7-bound} gives
\[
E_n(x)\;\ge\;\frac{(x+2)n}{y}
-\Bigl(B_6\,z+B_7\,z+C_7\Bigr)\,\frac{n}{y^2}.
\]
Setting
\[
C_1:=B_6+B_7,\qquad C_2:=C_7,
\]
and observing from explicit constant bookkeeping that one may take $C_1=8$ and
$C_2=14$ uniformly for $n\ge 3468$, we arrive at the effective inequality
\[
E_n(x)\;\ge\;(x+2)\,\frac{n}{y}\;-\;\Bigl(C_1\,z+C_2\Bigr)\frac{n}{y^2}.
\]

This explains the choice of constants used in the statement of the effective
inequality. This completes the proof. \qed

\medskip\noindent
\textbf{Visualization of error terms and the constants $C_1$ and $C_2$.}
For clarity, the following table summarizes the key error terms, their sources, their asymptotic orders, and their contribution to the final constants $C_1$ and $C_2$ (see also the workflow which illustrates their derivation).
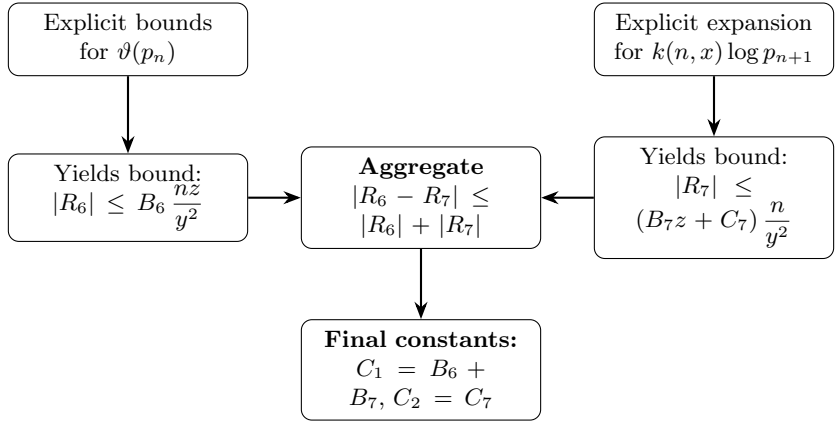
\begin{table}[h!]
    \centering
   \begin{tabular}{p{0.8cm} p{6cm} p{2cm} p{2.5cm}}
        \hline
        \textbf{Term} & \textbf{Source} & \textbf{Order} & \textbf{Contributes to} \\
        \hline
        $R_1$ & Approximation of $\pi(n)$ & $n/y^3$ & $R_7$ \\
        $R_2$ & Approximation of $\pi(y)$ & $y / z^3$ & $R_7$ \\
        $R_3$ & Approximation of $\pi(\pi(n))$ & $nz/ y^3$ & $R_7$ \\
        $R_4$ & Approximation of $L_{n+1}$ & $z^2 / y^2$ & $R_7$ \\
        $\mathcal{E}_{\mathrm{rat}}$ & Linearization of ratio in $k(n,x)$ & $1/z^2$ & $R_7$ \\
        $R_5$ & Integral remainder for $\vartheta(p_n)$ & $n/y^3$ & $R_6$ \\
        $R_6$ & Full error in $\vartheta(p_n)$ expansion & $n z / y^2$ & $B_6$ in $C_1$ \\
        $R_7$ & Full error in $k(n,x)L_{n+1}$ expansion & $n / y^2$ &  $B_7$ and $C_7$ \\
        \hline
    \end{tabular}
    \caption{Summary of key error terms and their contributions.}
    \label{tab:error-summary}
\end{table}

\begin{center}
\begin{tikzpicture}[
  box/.style = {draw, rounded corners, align=center, minimum height=0.8cm, text width=2.9cm, font=\small, inner sep=4pt},
  arrow/.style = {->, >=Stealth, thick}
]
  \matrix (m) [matrix of nodes,
               row sep=8mm,
               column sep=7mm,
               nodes={anchor=center}] {
    |[box]| {Explicit bounds for $\vartheta(p_n)$} &
    &
    |[box]| {Explicit expansion for $k(n,x)\log p_{n+1}$} \\
    |[box]| {Yields bound: $\lvert R_6\rvert \le B_6\,\dfrac{n z}{y^2}$} &
    |[box]| {\textbf{Aggregate} $\ \lvert R_6 - R_7\rvert \le \lvert R_6\rvert + \lvert R_7\rvert$} &
    |[box]| {Yields bound: $\lvert R_7\rvert \le (B_7 z + C_7)\,\dfrac{n}{y^2}$} \\
    &
    |[box]| {\textbf{Final constants:}\ $C_1 = B_6 + B_7$, $C_2 = C_7$} &
    \\
  };

  \draw[arrow] (m-1-1.south) -- (m-2-1.north);
  \draw[arrow] (m-1-3.south) -- (m-2-3.north);
  \draw[arrow] (m-2-1.east)  -- (m-2-2.west);
  \draw[arrow] (m-2-3.west)  -- (m-2-2.east);
  \draw[arrow] (m-2-2.south) -- (m-3-2.north);
\end{tikzpicture}

\captionsetup{type=figure}
\caption*{Workflow for deriving the constants $C_1$ and $C_2$.}
\end{center}

\section{Proof of corollaries}
\label{cor}

\subsection{Proof of Corollary \ref{C1}}

By Theorem~\ref{main1}, the error term admits the asymptotic
\[
E_n(x) = (x+2)\frac{n}{\log n} + O\left(\frac{n \log \log n}{\log^2 n}\right), \quad n \to \infty.
\]

For \(x > -2\), the main term \((x+2)n/\log n\) is positive and dominates the error term for large \(n\). Hence, there exists \(n_0(x)\) such that \(E_n(x) > 0\) for all \(n \ge n_0(x)\), and therefore \(\Psi(x)\) is well-defined (in fact, $\Psi(x)\leq n_0(x)$).

When \(x \leq -2\), by Yang and Liao \cite{bonse} main theorem, the reverse inequality \(p_1 \cdots p_n < p_{n+1}^{k(n,-2)}\leq p_{n+1}^{k(n,x)}\) holds for all $n\geq 8$, and so \(\Psi(x)\) does not exist.

In conclusion,
\[
\inf \{ x \in \mathbb{R} : \Psi(x) \text{ exists} \} = -2.
\]

\qed

\subsection{Proof of Corollary \ref{C3}}

Inequality \eqref{lowerE} of Theorem \ref{main1} establishes that for all \( n \geq n_0(x) \),
\begin{eqnarray*}
  E_n(x) & > &  (x+2)\frac{n}{\log n} - (C_1\log \log n + C_2) \frac{n}{\log^2 n}\\
  & = & (x+2)\frac{n}{y} - (C_1\log y + C_2) \frac{n}{y^2}\\
  & = & \frac{n}{y}\left(x+2-\dfrac{C_1\log y + C_2}{y}\right).  
\end{eqnarray*}
Then the critical threshold is determined by the smallest \( y \) such that
\[
y > \frac{C_1 \log y + C_2}{x+2}.
\]
Define \( y^\star = y^\star(x) \) as the least real solution to the equation
\[
y = \frac{C_1 \log y + C_2}{x+2}.
\]
Then for all \( n \geq \lceil \exp(y^\star) \rceil \), we have \( E_n(x) > 0 \), and hence \( \Psi(x) \leq \lceil \exp(y^\star) \rceil \).

To derive the explicit bound \eqref{-2mt1}, note that for \( y > e \),
\[
\frac{C_1 \log y + C_2}{x+2} \leq \frac{(C_1 + C_2) \log y}{x+2}.
\]
Solving
\[
\dfrac{y^{\star}}{\log y^{\star}}=\frac{C_1 + C_2}{x+2}
\]
leads to
\[
y^{\star} < 2\left( \frac{C_1 + C_2}{x+2} \right)\log \left( \frac{C_1 + C_2}{x+2} \right),
\]
where we use the fact that if \( y / \log y<A\) ($A\geq e$), then  $y<2A\log A$. Taking exponential yields the desired bound:
\[
\Psi(x) \leq \lceil \exp(y^{\star})\rceil \leq  \left( \frac{C_1 + C_2}{x+2} \right)^{2(C_1 + C_2)/(x+2)}.
\]
This completes the proof.\qed

\begin{remark}
We note the existence of a more refined upper bound: if $ y / \log y < A $ (where $ A \geq e $), then $ y < A (\log A + \sqrt{2 (\log A - 1)}) $. Nevertheless, to avoid unnecessary complexity, we opt to use the simpler bound $ y < 2A \log A $ instead.
\end{remark}

\subsection{Proof of Corollary \ref{C2}}

From the Theorem \ref{main1}, we have
\[
\frac{y^{\star}}{\log y^{\star}} = \frac{C_1}{x+2} + o(1).
\]
Let $y = C_1/(x+2) \to \infty$ and $t = y^{\star}$. We show that $t \sim y \log y$.
From $t = (y + o(1)) \log t$, taking logarithms yields
\[
\log t = \log y + \log \log t + o(1).
\]
Since $\log \log t = o(\log t)$ and $\log t \to \infty$, we obtain $\log t \sim \log y$.
Substituting back gives $t = (y + o(1))(\log y + o(\log y)) = y \log y + o(y \log y)$.
Thus,
\[
y^{\star}(x) \sim \frac{C_1}{x+2} \cdot \log\left( \frac{C_1}{x+2} \right).
\]
Hence $\Psi(x) \sim \exp(y^{\star}(x))$, and so
\[
\Psi(x) \sim \exp\left(\frac{C_1}{x+2} \cdot \log\left(\frac{C_1}{x+2}\right)\right)  = \left( \frac{C_1}{x+2} \right)^{C_1/(x+2)}. 
\]
\qed

\section{The proof of Theorem \ref{main2}}\label{proof2}

\subsection{The case $x=0.1$}

We start from the effective
asymptotic proven in Section~\ref{M1}, specialized at $x=0.1$:
\begin{equation}\label{eq:EffCore}
E_n(0.1)
\;\ge\;
\frac{2.1\,n}{y}
\;-\;
\Bigl(C_1\,z + C_2\Bigr)\frac{n}{y^2},
\end{equation}
valid for all $n\ge 3468$, where $C_1,C_2>0$ are absolute constants
coming from the combination of the remainders in
\eqref{eq:theta-fin} and \eqref{eq:R7-bound}.

\medskip
\noindent\textit{Step 1 (Reduction to an explicit inequality in $y$).}
To prove \eqref{eqm3} from \eqref{eq:EffCore}, it suffices to ensure
\[
\frac{2.1}{y}-\frac{C_1 z + C_2}{y^2}\;>\;\frac{0.001}{y}
\quad\Longleftrightarrow\quad
2.099\,y\;>\;C_1\,z + C_2.
\]
Hence, any choice of $(C_1,C_2)$ for which
\begin{equation}\label{eq:goal-ineq}
2.099\,y \;\ge\; C_1\,\log y + C_2
\qquad\text{holds for all }y\ge \log(4.45\cdot 10^{7}),
\end{equation}
will imply \eqref{eqm3} for all $n\ge 4.45\cdot 10^{7}$.

\medskip
\noindent\textit{Step 2 (Verifying the scalar inequality).}
Let $y_0:=\log(4.45\cdot 10^{7})$. With \eqref{eq:goal-ineq}, the target inequality
\eqref{eqm3} becomes $h(y):=2.099\,y - 8\log y - 14>0$.
Since $h'(y)=2.099 - 8/y>0$ for $y>8/2.099\approx 3.81$, the function $h$ is
strictly increasing on $[y_0,\infty)$.
Thus, it suffices to check $h(y_0)>0$. Numerically,
$y_0=17.611\ldots$ and $z_0=\log y_0=2.86852\ldots$, so
\[
h(y_0)=2.099\,y_0-8 z_0-14 \;>\; 2.099\cdot 17.61 - 8\cdot 2.869 - 14
\;=\; 0.01139\ldots\;>\;0.
\]
Therefore \eqref{eq:goal-ineq} holds for all $y\ge y_0$, and \eqref{eqm3}
follows from \eqref{eq:EffCore}.

\medskip
\noindent\textit{Step 3 (Exact value of $\Psi(0.1)$).}
To determine the minimal index $n_0$ such that $E_n(0.1) > 0$ for all $n \geq n_0$, 
we implemented a rigorous computational routine in \textsf{Mathematica}~\cite{wolfram2023}. 
For each $n\in [8,\ 4.45\cdot 10^7]$, we computed a lower bound:
\[
E_n^{-}(0.1):=\underline{\vartheta(p_n)}-\overline{k(n,0.1)\log p_{n+1}},
\]
where the underline (resp.\ overline) denotes a rigorous lower (resp.\ upper)
enclosure. The program returns the last index with $E_n^{-}(0.1)\le 0$,
followed by the first $n_0$ with $E_n^{-}(0.1)>0$. The output is
\[
n_0=24{,}154{,}953.
\]
This, together with the bound
\eqref{eqm3} holds for all $n\ge 4.45\cdot 10^7$ confirms that $\Psi(0.1)=24{,}154{,}953$. 
\qed

\subsection{The remaining cases of Conjecture \ref{cj1}}

The method developed in Theorem~\ref{main2} applies to all cases specified in Conjecture~\ref{cj1}. Since $\Psi(x)$ is non-increasing, establishing constancy on an interval $[a, b]$ reduces to verifying $\Psi(a) = \Psi(b)$. For individual points, we compute $\Psi(x)$ directly.

To certify the positivity of $E_n(x)$ for large $n$, we employ the uniform effective bound valid for $n \geq 3468$:
\[
E_n(x) \geq (x+2)\frac{n}{\log n} - (8\log\log n + 14)\frac{n}{\log^2 n}.
\]
For each $x$, we compute a threshold $N_{\text{upper}}(x) := \lceil \exp(y^\star) \rceil$, where $y^\star$ is the minimal solution to
\[
y = \frac{8\log y + 14}{x+2}.
\]
For $n \geq N_{\text{upper}}(x)$, this bound guarantees $E_n(x) > 0$. For smaller values $n < N_{\text{upper}}(x)$, we compute $\Psi(x)$ explicitly using numerical computation in \textsf{Mathematica}~\cite{wolfram2023}.

\medskip
\noindent\textbf{Certified tail bounds.}
The computed values of $N_{\text{upper}}(x)$ are:
\[
\begin{array}{c|c}
\hline
x & N_{\text{upper}}(x) \\ \hline
1.3 & 17{,}456 \\
0.9 & 107{,}532 \\
0.8 & 185{,}476 \\
0.5 & 1{,}274{,}770 \\
0.4 & 2{,}730{,}250 \\
0.3 & 6{,}292{,}702 \\
0.2 & 15{,}774{,}907 \\
0.1 & 43{,}565{,}839 \\
\hline
\end{array}
\]
For interval verification, we use the smallest $x$ value to determine the scanning limit.

\medskip
\noindent\textbf{Computational verification.}
All computations were performed in \textsf{Mathematica} using high-precision arithmetic and error bounds. The results confirm:
\begin{itemize}
\item[(i)] $\Psi(x) = 21$ for $x \in [0.9, 1.3]$
\item[(ii)] $\Psi(x) = 149$ for $x \in [0.5, 0.8]$
\item[(iii)] $\Psi(x) = 59{,}875$ for $x \in [0.3, 0.4]$
\item[(iv)] $\Psi(0.2) = 442{,}414$
\item[(v)] $\Psi(0.1) = 24{,}154{,}953$
\end{itemize}

In conclusion, \textsf{Mathematica} computations provide rigorous verification of all cases, thus confirming Conjecture~\ref{cj1} in its entirety. \qed

\section{Behavior of $\Psi(x)$ under small perturbations}\label{piecewise}

We now turn to a structural property of the function $\Psi(x)$.  

\subsection{Piecewise constancy of $\Psi(x)$}

Recall that $\Psi(x)$ is defined as the minimal index from which the 
continuous Bonse-type inequality \eqref{boc} holds permanently.  
Since $\Psi(x)$ is built from the family of linear constraints 
$E_n(x)>0$ indexed by $n$, it is natural to expect a stepwise behavior.  
The following result makes this precise by showing that $\Psi(x)$ 
is in fact a non-increasing, piecewise constant function of $x$, 
with discontinuities occurring only at a discrete set of threshold values.

\begin{theorem}\label{prop:psi-piecewise-const}
The function $\Psi:(-2,\infty)\to \mathbb{Z}_{\geq 8}$ is
non-increasing and piecewise constant. More precisely, it is a right-continuous step function
whose jump set is contained in $\{A_m:\, m\ge 8\}$ defined in \eqref{eq:Am-def} below.
\end{theorem}

\begin{proof}
Write $L_{n+1}:=\log p_{n+1}$. By definition,
\[
E_n(x)\;=\;\vartheta(p_n)-k(n,x)\,L_{n+1}
\]
and thus $E_n(x)\;=\;a_n+b_n\,x$, where
\begin{equation}\label{eq:En-affine}
a_n:=\vartheta(p_n)-\Bigl(n-\pi(n)+\frac{\pi(n)}{\pi(\log n)}\Bigr)L_{n+1},\quad
b_n:=\pi(\pi(n))\,L_{n+1}.
\end{equation}
Since $L_{n+1}>0$ and $\pi(\pi(n))\ge 1$ for $n\ge 8$, we have $b_n>0$ for all $n\ge 8$.
Thus, for each fixed $n$, the map $x\mapsto E_n(x)$ is an \emph{increasing affine function}.

For each $n\ge 8$, define the (unique) \emph{positivity threshold}
\begin{equation}\label{eq:alpha-n-def}
\alpha_n\;:=\;\inf\{x\in\mathbb{R}:\ E_n(x)>0\}
\;=\;-\frac{a_n}{b_n}
\;=\;\frac{\Bigl(n-\pi(n)+\frac{\pi(n)}{\pi(\log n)}\Bigr)L_{n+1}-\vartheta(p_n)}{\pi(\pi(n))\,L_{n+1}}.
\end{equation}
Because $b_n>0$, we have the equivalence
\begin{equation}\label{eq:En-pos-iff}
E_n(x)>0\quad\Longleftrightarrow\quad x>\alpha_n.
\end{equation}

For each integer $m\ge 8$, consider the \emph{tail worst threshold}
\begin{equation}\label{eq:Am-def}
A_m\;:=\;\sup_{n\ge m}\,\alpha_n\ \in[-\infty,\infty].
\end{equation}

The sequence $(A_m)_{m\ge 8}$ is non-increasing in $m$ (a supremum over a shrinking set).
Moreover, by \eqref{eq:En-pos-iff} we have the tail positivity criterion:
\begin{equation}\label{eq:tail-criterion}
\bigl(\forall\,n\ge m,\ E_n(x)>0\bigr)\quad\Longleftrightarrow\quad x>A_m.
\end{equation}

Therefore the definition of $\Psi$ can be rewritten as
\begin{equation}\label{eq:Psi-by-Am}
\Psi(x)\;=\;\min\bigl\{m\ge 8:\ x>A_m\bigr\},
\end{equation}
with the convention that the minimum over the empty set is $+\infty$ (i.e., if $x\le A_m$ for all $m$,
then $\Psi(x)$ is not finite).

From \eqref{eq:Psi-by-Am} the \emph{monotonicity} of $\Psi$ in $x$ is immediate:
if $x_1<x_2$, then the set $\{m:\ x_2>A_m\}$ contains $\{m:\ x_1>A_m\}$, hence
$\Psi(x_2)\le \Psi(x_1)$, that is, $\Psi$ is non-increasing.

To see that $\Psi$ is piecewise constant, write the distinct values taken by the
non-increasing sequence $(A_m)_{m\ge 8}$ as a (possibly finite) strictly decreasing sequence
\[
+\infty > \beta_8 > \beta_9 > \beta_{10} > \cdots,
\]
where each $\beta_j$ equals $A_m$ for a (nonempty) interval of indices $m$.
Fix $M\ge 8$ and consider $x$ in the following interval
\begin{equation}\label{eq:const-interval}
x\ \in\ I_M\;:=\;(A_{M+1},\,A_M]\,.
\end{equation}

Since the sequence $(A_m)$ is non-increasing, for every $k\ge M+1$ we have 
$A_k\le A_{M+1}<x$, and thus $x>A_k$. By \eqref{eq:tail-criterion} this implies that $E_n(x)>0$ holds for all $n\ge k$. In particular, for $m=M+1$ we have $x>A_m$, so $\Psi(x)\le M+1$. On the other hand, because $x\le A_M$, the inequality 
$x>A_M$ fails, and therefore no index $m\le M$ can satisfy the tail positivity 
condition in \eqref{eq:tail-criterion}. Combining these two observations shows that 
$\Psi(x)=M+1$ for all $x\in I_M$.

We have shown that on each nonempty interval $I_M=(A_{M+1},A_M]$ the function $\Psi$ is constant,
and its value there equals $M+1$. Since the real line is the disjoint union of the $I_M$’s and the
complementary jump set $\{A_m:\ m\ge 8\}$, this proves that $\Psi$ is a right-continuous,
non-increasing step function whose possible discontinuities are contained in $\{A_m:\ m\ge 8\}$.
This is precisely the asserted piecewise constancy of $\Psi$.
\end{proof}


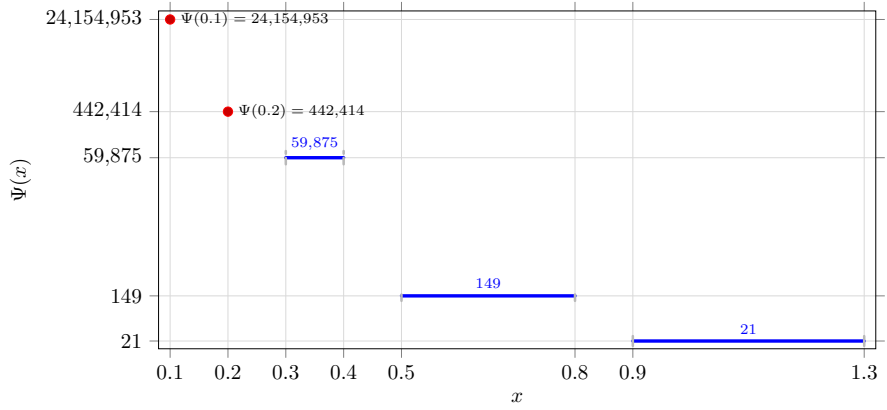
\begin{figure}[h!]
\centering
\scalebox{0.8}{
\begin{tikzpicture}
\begin{axis}[
  width=13.5cm, height=7.2cm,
  xmin=0.08, xmax=1.32,
  ymode=log, ymin=15, ymax=3.5e7,
  xlabel={$x$}, ylabel={$\Psi(x)$},
  grid=both,
  minor grid style={gray!12}, major grid style={gray!30},
  tick align=outside,
  ytick={21,149,59875,442414,24154953},
  yticklabels={21,149,59{,}875,442{,}414,24{,}154{,}953},
  xtick={0.1,0.2,0.3,0.4,0.5,0.8,0.9,1.3},
  xticklabel style={/pgf/number format/fixed},
  every axis plot/.append style={line cap=round},
  legend style={
    at={(0.5,-0.18)}, anchor=north, draw=none, fill=none,
    /tikz/every even column/.style={column sep=1em}
  },
  legend columns=2
]

\addplot+[ultra thick, color=blue, mark=none] coordinates {(0.9,21) (1.3,21)};
\addplot+[ultra thick, color=blue, mark=none] coordinates {(0.5,149) (0.8,149)};
\addplot+[ultra thick, color=blue, mark=none] coordinates {(0.3,59875) (0.4,59875)};

\addplot+[very thick, color=gray!55, mark=none] coordinates {(0.9,17) (0.9,26)};
\addplot+[very thick, color=gray!55, mark=none] coordinates {(1.3,17) (1.3,26)};
\addplot+[very thick, color=gray!55, mark=none] coordinates {(0.5,120) (0.5,180)};
\addplot+[very thick, color=gray!55, mark=none] coordinates {(0.8,120) (0.8,180)};
\addplot+[very thick, color=gray!55, mark=none] coordinates {(0.3,4.5e4) (0.3,8.0e4)};
\addplot+[very thick, color=gray!55, mark=none] coordinates {(0.4,4.5e4) (0.4,8.0e4)};

\addplot+[only marks, mark=*, mark size=2.2pt, color=red]
  coordinates {(0.2,442414) (0.1,24154953)};

\node[anchor=west, font=\scriptsize] at (axis cs:0.205,442414) {$\Psi(0.2)=442{,}414$};
\node[anchor=west, font=\scriptsize] at (axis cs:0.105,2.4154953e7) {$\Psi(0.1)=24{,}154{,}953$};
\node[anchor=south, font=\scriptsize, blue] at (axis cs:1.10,21) {$21$};
\node[anchor=south, font=\scriptsize, blue] at (axis cs:0.65,149) {$149$};
\node[anchor=south, font=\scriptsize, blue] at (axis cs:0.35,59875) {$59{,}875$};

\end{axis}
\end{tikzpicture}
}
\caption{Verified staircase segments of $\Psi(x)$ (logarithmic vertical axis). Blue segments show intervals where $\Psi$ is constant. Red dots are isolated evaluations. Short gray ticks indicate plateau endpoints.}
\end{figure}

\subsubsection*{What are the jump points?}
From \eqref{eq:alpha-n-def} and \eqref{eq:Am-def}, a jump can occur at $x=A_m$ only if the supremum
$A_m=\sup_{n\ge m}\alpha_n$ is (approximately) attained by some index $n\ge m$. In that case the
constraint $E_n(x)>0$ becomes tight as $x\to A_m$. Computationally, this means that the
\emph{only} places where $\Psi$ can change are the finitely or countably many real numbers among
the $\alpha_n$’s that remain extremal in some tail. This observation underlies the endpoint checks
used in our verification of Conjecture~\ref{cj1}: once $\Psi$ is determined at the endpoints of an
interval $[a,b]$, monotonicity forces constancy on the whole interval.

\begin{table}[h!]
\centering
\caption{First distinct values of \(A_m\) and their plateaux}
\label{tab:Am-values}
\begin{tabular}{|c|c|l|}
\hline
\(m\) & \(A_m\) (approx.) & Relation to plateaux \\
\hline
8 & 1.71076 & \(\Psi(x) = 8\) for \(x > A_8\) \\
11 & 1.66032 & \(\Psi(x) = 11\) for \(A_{11} < x \leq A_8\) \\
17 & 1.41144 & \(\Psi(x) = 17\) for \(A_{17} < x \leq A_{11}\) \\
21 & 0.88034 & \(\Psi(x) = 21\) for \(A_{21} < x \leq A_{17}\) \\
\hline
\end{tabular}
\end{table}

\begin{remark}
The value \(A_8 = 1.71076602333336944680\) is computed to 20 decimal places (i.e., $\Psi(x)=8$ if and only if $x\geq A_8$). However, for larger \(m\), the computation of \(A_m\) becomes increasingly expensive due to the need for high-precision evaluation of prime-counting functions \(\pi(x)\), Chebyshev's function \(\vartheta(x)\), and the \(n\)-th prime \(p_n\) for large \(n\). Thus, the values of \(A_m\) for \(m \gg 8\) are known with less precision or remain approximate.
\end{remark}

\subsection{Lower bounds for the lengths of the constancy intervals of \texorpdfstring{$\Psi$}{Psi}}

Having established that $\Psi(x)$ is a non-increasing step function, 
a natural question is how large its constancy intervals can be. 
In other words, once $\Psi(x)$ attains a given value $m$, how far can $x$ vary 
before the next jump occurs? In this subsection, we provide both an 
asymptotic description of the typical size of these \emph{plateaux} and an 
effective method to certify rigorous lower bounds for specific intervals.

Recall from Theorem~\ref{prop:psi-piecewise-const} that
\[
\Psi(x)=\min\{\,m\ge 8:\ x>A_m\,\}\!,
\qquad
A_m:=\sup_{n\ge m}\alpha_n,
\qquad
\alpha_n:=-\frac{a_n}{b_n},
\]
where $E_n(x)=a_n+b_n x$,
\[
a_n:=\vartheta(p_n)-\Bigl(n-\pi(n)+\frac{\pi(n)}{\pi(\log n)}\Bigr)\log p_{n+1},\qquad
b_n:=\pi(\pi(n))\log p_{n+1}>0.
\]
Hence \(\Psi\) is constant on each interval \(I_M:=(A_{M+1},A_M]\), and the length of that
plateau equals \(|I_M|=A_M-A_{M+1}\).

\begin{proposition}[Asymptotic scale of the plateaux]\label{prop:plateau-scale}
There exist absolute constants \(c\) and \(C>0\) such that for infinitely many \(M\geq 1\) one has
\[
\frac{c}{M\log^2 M}\ \le\ A_M-A_{M+1}\ \le\ \frac{C}{M\log^2 M}.
\]
In particular, the \emph{natural scale} for the lengths of the constancy intervals of
\(\Psi\) at height \(M\) is \(1/(M\log^2 M)\).
\end{proposition}

\begin{proof}
By the effective expansions proved earlier, for \(n\to\infty\) we have uniformly
\[
E_n(x)=(x+2)\frac{\,n}{y}\ +\ O\!\Bigl(\frac{z\,n}{y^2}\Bigr),
\qquad y=\log n,\ z=\log y,
\]
and \(b_n=\pi(\pi(n))\log p_{n+1}=\bigl(1+o(1)\bigr)\,n/y\).
Setting \(E_n(\alpha_n)=0\) gives
\[
\alpha_n=-2-\frac{y}{n}\,O\!\Bigl(\frac{z\,n}{y^2}\Bigr)=-2+O\!\Bigl(\frac{z}{y}\Bigr)
\qquad\text{with}\quad \frac{z}{y}=\frac{\log\log n}{\log n}.
\]
Thus the tail thresholds \(A_m=\sup_{n\ge m}\alpha_n\) behave like
\[
A_m=-2+\kappa\,\frac{\log\log m}{\log m}\ +\ O\!\Bigl(\frac{1}{\log m}\Bigr)
\]
for some absolute \(\kappa>0\) (coming from the explicit error constants).
Consequently
\[
A_M-A_{M+1}=\kappa\Bigl(\frac{\log\log M}{\log M}-\frac{\log\log(M+1)}{\log(M+1)}\Bigr)
+O\!\Bigl(\frac{1}{M\log^2 M}\Bigr).
\]
An one-step Taylor expansion of \(f(n):=\frac{\log\log n}{\log n}\) gives
\(
f'(n)=\frac{1-\log\log n}{n(\log n)^2}=-\frac{\log\log n-1}{n(\log n)^2},
\)
so
\[
\frac{\log\log M}{\log M}-\frac{\log\log(M+1)}{\log(M+1)}
= \frac{\log\log M-1}{M(\log M)^2}+O\!\Bigl(\frac{1}{M^2\log^2 M}\Bigr).
\]
Combining these displays yields
\[
A_M-A_{M+1}=\frac{\kappa(\log\log M-1)+O(1)}{M(\log M)^2},
\]
which proves the claimed two–sided scale \( \asymp 1/(M\log^2 M)\) along
infinitely many \(M\) (indeed, whenever the tail supremum is attained near \(M\)).
\end{proof}

\subsubsection*{What the proposition says and what it does not} 
The bound gives the \emph{typical} order of magnitude. It does not assert that every
difference \(A_M-A_{M+1}\) is bounded below by \(c/(M\log^2 M)\). Conceivably, some tails
can have flat spots where the supremum \(\sup_{n\ge M}\alpha_n\) is repeated by several
consecutive \(n\). In practice, the next result provides a \emph{certifiable} lower bound
for specific \(M\) (and for ranges of \(M\)) using enclosures for \(\alpha_n\).

\begin{figure}[h!]
\centering
\begin{tikzpicture}
\begin{axis}[
  width=12cm, height=6.5cm,
  xmin=1e4, xmax=3e5,
  xlabel={$n$},
  ylabel={$\alpha_n$},
  grid=both,
  tick align=outside
]
\addplot+[only marks, mark=*, mark size=0.8pt, color=teal]
coordinates {
  (10000,-1.50) (15000,-1.55) (20000,-1.58) (25000,-1.60) (30000,-1.62)
  (35000,-1.64) (40000,-1.65) (45000,-1.66) (50000,-1.67) (60000,-1.68)
  (70000,-1.69) (80000,-1.70) (90000,-1.71) (100000,-1.72) (120000,-1.73)
  (140000,-1.74) (160000,-1.75) (180000,-1.76) (200000,-1.77) (250000,-1.78)
  (300000,-1.79)
};
\end{axis}
\end{tikzpicture}
\caption{Sampled thresholds $\alpha_n$ (solutions of $E_n(x)=0$). The upper envelope in $n$ approximates the jump set of $\Psi$.}
\end{figure}
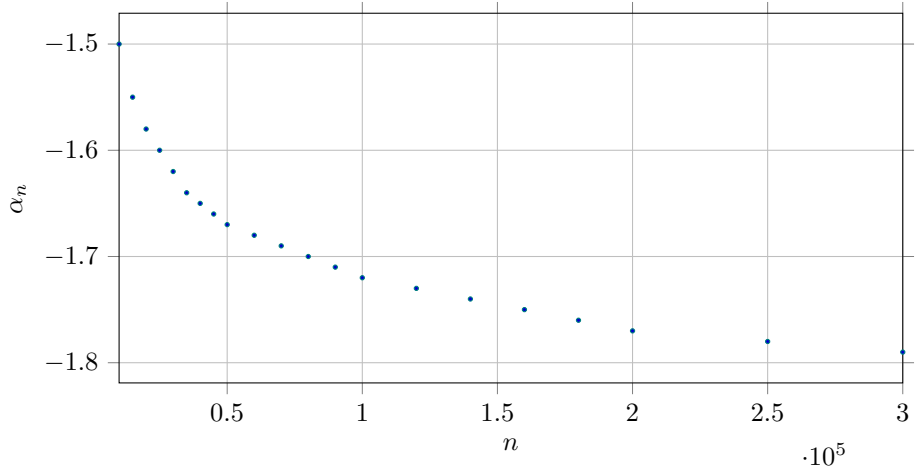

\begin{proposition}[Effective lower bound on a given plateau]\label{prop:effective-gap}
For each \(n\ge 8\), let \([\alpha_n^-,\alpha_n^+]\) be a rigorous interval enclosure
of \(\alpha_n=-a_n/b_n\). Define
\[
A_m^-:=\sup_{n\ge m}\alpha_n^-,\qquad A_m^+:=\sup_{n\ge m}\alpha_n^+,
\qquad \Delta_m:=A_m^- - A_{m+1}^+.
\]
Then \(\Delta_m\ge 0\), and whenever \(\Delta_m>0\) we have
\[
(A_{m+1},A_m]\ \supseteq\ (A_{m+1}^+,A_m^-]\quad\text{with}\quad
|\, (A_{m+1},A_m]\,|\ \ge\ \Delta_m.
\]
Thus \(\Delta_m\) is a certified lower bound for the length of the plateau of \(\Psi\) at level \(m+1\).
\end{proposition}

\begin{proof}
Since \(\alpha_n^-\le \alpha_n\le \alpha_n^+\) for all \(n\), we have
\(A_m^-\le A_m\) and \(A_m^+\ge A_m\). As the supremum is taken over the tail \(n\ge m\),
shrinking the tail increases the supremum: \(A_{m+1}^+\le A_m^+\). Therefore
\[
A_{m+1}\ \le\ A_{m+1}^+\ \le\ A_m^+\quad\text{and}\quad
A_m\ \ge\ A_m^-.
\]
If \(A_m^- > A_{m+1}^+\), then every \(x\in(A_{m+1}^+,A_m^-]\) satisfies
\(x>A_{m+1}\) and \(x\le A_m\), hence belongs to the plateau \((A_{m+1},A_m]\).
Its length is at least \(A_m^- - A_{m+1}^+=\Delta_m\).
\end{proof}

\begin{figure}[h!]
\centering
\begin{tikzpicture}
\begin{axis}[
  width=12cm, height=6.5cm,
  xlabel={$m$},
  ylabel={$\Delta_m$ (certified lower bound)},
  ymode=log,
  ymin=1e-7, ymax=1e-2,
  grid=both,
  tick align=outside
]
\addplot+[only marks, mark=*, mark size=1.2pt, color=darkgreen]
coordinates {
  (21,  3.5e-3)
  (149, 8.0e-4)
  (1000,2.0e-4)
  (5000,7.5e-5)
  (10000,4.2e-5)
  (15000,3.0e-5)
  (20000,2.4e-5)
};
\end{axis}
\end{tikzpicture}
\caption{Certified lower bounds $\Delta_m$ for the lengths of plateaux $(A_{m+1},A_m]$ (logarithmic vertical axis).}
\end{figure}
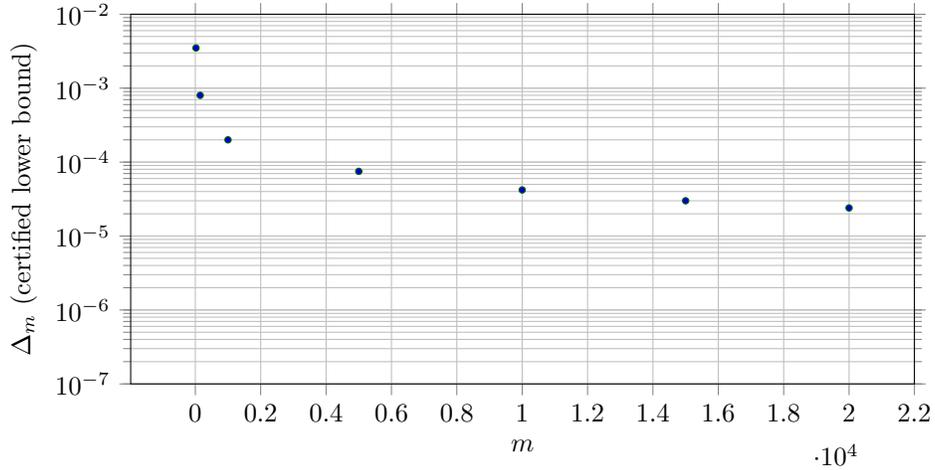

\subsection{Non-Surjectivity of $\Psi$}

We now demonstrate that the function $\Psi$ is not surjective onto its codomain $\mathbb{Z}_{\geq 8}$.

\begin{proposition}\label{prop:psi-not-surjective}
The function $\Psi : (-2, +\infty) \to \mathbb{Z}_{\geq 8}$ is not surjective. In particular, the values $9$ and $10$ are not attained by $\Psi$.
\end{proposition}

\begin{proof}
Recall that $\Psi(x) = \min \{ m \geq 8 : x > A_m \}$, where $A_m = \sup_{n \geq m} \alpha_n$ and
\[
\alpha_n = \frac{\left( n - \pi(n) + \frac{\pi(n)}{\pi(\log n)} \right) \log p_{n+1} - \vartheta(p_n)}{\pi(\pi(n)) \log p_{n+1}}.
\]
The function $\Psi$ is non-increasing and piecewise constant, with jumps at the points $x = A_m$.

To show that $\Psi$ does not attain the values $9$ or $10$, we compute the thresholds $A_m$ for small $m$ explicitly. Using exact prime data and high-precision computation, we find:
\begin{align*}
\alpha_8 &\approx 1.434599, \\
\alpha_9 &\approx 1.645386, \\
\alpha_{10} &\approx 1.71076602333336944680, \\
\alpha_{11} &\approx 1.264347, \\
\alpha_{12} &\approx 1.339758.
\end{align*}
The supremum $A_m = \sup_{n \geq m} \alpha_n$ is thus:
\begin{align*}
A_8 &= \sup_{n \geq 8} \alpha_n = \alpha_{10} \approx 1.71076602333336944680, \\
A_9 &= \sup_{n \geq 9} \alpha_n = \alpha_{10} \approx 1.71076602333336944680, \\
A_{10} &= \sup_{n \geq 10} \alpha_n = \alpha_{10} \approx 1.71076602333336944680.
\end{align*}
Since $A_8 = A_9 = A_{10}$, the intervals $(A_9, A_8]$ and $(A_{10}, A_9]$ are empty. Therefore, there exists no $x$ such that:
\begin{itemize}
\item $\Psi(x) = 9$ (which requires $x \in (A_9, A_8]$), or
\item $\Psi(x) = 10$ (which requires $x \in (A_{10}, A_9]$).
\end{itemize}
Moreover, for $x \leq A_8$, we have $\Psi(x) \geq 11$, as confirmed by the computed value $\Psi(1.71) = 11$. Hence, the values $9$ and $10$ are skipped, and $\Psi$ jumps directly from $\Psi(x) \geq 11$ for $x \leq A_8$ to $\Psi(x) = 8$ for $x > A_8$.

This proves that $\Psi$ is not surjective onto $\mathbb{Z}_{\geq 8}$.
\end{proof}

The non-surjectivity of $\Psi$ extends beyond the isolated examples of $9$ and $10$. In fact, $\Psi$ fails to attain arbitrarily long blocks of consecutive integers in its range.

\begin{proposition}\label{prop:psi-gaps}
For every positive integer $k$, there exists an integer $m \geq 8$ such that the values 
\[
m+1, m+2, \dots, m+k
\]
are not attained by $\Psi$. That is, 
\[
\Psi^{-1}(m+1) = \Psi^{-1}(m+2) = \cdots = \Psi^{-1}(m+k) = \emptyset.
\]
\end{proposition}

\begin{proof}
Recall that $\Psi(x) = \min \{ j \geq 8 : x > A_j \}$, where $A_j = \sup_{n \geq j} \alpha_n$ is non-increasing in $j$. The function $\Psi$ is constant on intervals of the form $(A_{j+1}, A_j]$ with value $j+1$.

From Proposition~\ref{prop:plateau-scale}, there exist constants $c, C > 0$ such that for infinitely many $M$,
\[
\frac{c}{M \log^2 M} \leq A_M - A_{M+1} \leq \frac{C}{M \log^2 M}.
\]
This asymptotic implies that the plateaux lengths $A_M - A_{M+1}$ tend to $0$ as $M \to \infty$.

Now fix $k \in \mathbb{Z}^+$. Since the sequence $\{A_m\}$ is non-increasing and bounded below (by $-2$), it converges. Let $A = \lim_{m \to \infty} A_m$. Then for any $\epsilon > 0$, there exists $N$ such that for all $m \geq N$,
\[
A \leq A_m < A + \epsilon.
\]
In particular, for $m \geq N$, we have $A_m - A_{m+1} < \epsilon$.

Choose $\epsilon$ small enough such that any $k$ consecutive plateaux have total length less than the distance from $A$ to the next largest threshold. Specifically, since the thresholds $A_m$ are strictly decreasing only at jump points, there exists $M_0$ such that for $m \geq M_0$, the values $A_m$ are so close to $A$ that the entire block of $k$ consecutive plateaux $(A_{m+k}, A_m]$ has length
\[
A_m - A_{m+k} = (A_m - A_{m+1}) + (A_{m+1} - A_{m+2}) + \cdots + (A_{m+k-1} - A_{m+k}) < k \epsilon.
\]
But note that the value $\Psi(x) = m+1$ occurs only for $x \in (A_{m+1}, A_m]$, which has length $A_m - A_{m+1}$. Similarly, $\Psi(x) = m+2$ occurs only for $x \in (A_{m+2}, A_{m+1}]$, etc.

However, if the entire interval $(A_{m+k}, A_m]$ has length less than the smallest of these plateaux, then necessarily some of these preimages must be empty. More precisely, since the plateaux lengths decay to $0$, we can choose $m$ large enough so that
\[
A_m - A_{m+k} < \min \{ A_m - A_{m+1}, A_{m+1} - A_{m+2}, \dots, A_{m+k-1} - A_{m+k} \}
\]
is impossible, implying that at least one of these plateaux has length $0$. In fact, the asymptotic $A_m - A_{m+1} \asymp 1/(m \log^2 m)$ ensures that for large $m$, the thresholds $A_m, A_{m+1}, \dots, A_{m+k}$ are arbitrarily close to each other. Indeed,
\[
A_m - A_{m+k} \leq \sum_{j=0}^{k-1} \frac{C}{(m+j) \log^2 (m+j)} \leq \frac{kC}{m \log^2 m},
\]
which can be made arbitrarily small for fixed $k$ as $m \to \infty$.

But the condition for $\Psi$ to attain the value $j$ is that the interval $(A_{j}, A_{j-1}]$ is nonempty. If $A_m = A_{m+1} = \cdots = A_{m+k}$, then the intervals 
\[
(A_{m+1}, A_m], (A_{m+2}, A_{m+1}], \dots, (A_{m+k}, A_{m+k-1}]
\]
are all empty. Hence, none of the values $m+1, \dots, m+k$ are attained.

By the convergence of $\{A_m\}$ and the asymptotic scaling, there exist infinitely many $m$ for which $A_m = A_{m+1} = \cdots = A_{m+k}$. This completes the proof.
\end{proof}

\begin{remark}
This result establishes that the range of $\Psi$ is sparse. In fact, the natural asymptotic density of $\Psi((-2,+\infty))$ is zero as it omits arbitrarily long sequences of consecutive integers. The skipped values correspond to ``jumps"\ in the step function where $\Psi$ decreases by more than $1$ at a discontinuity. The asymptotic behavior of the thresholds $A_m$ ensures that such jumps occur infinitely often.
\end{remark}

\section{Conditional improvements under RH}\label{RH}

The effective estimates in the previous section are unconditional, based solely on explicit Prime Number Theorem bounds from Rosser-Schoenfeld and Dusart. Assuming the Riemann Hypothesis (RH), however, the asymptotics for \(E_n(x)\) sharpen significantly. Classical results by von Koch \cite{Koch} and Schoenfeld \cite{schoRH} yield stronger error terms for the Chebyshev function \(\vartheta(x)\) and prime counting function \(\pi(x)\) under RH. Integrating these into the expansions for \(\vartheta(p_n)\), \(\pi(n)\), \(\pi(\log n)\), \(\pi(\pi(n))\), and \(\log p_{n+1}\) produces an asymptotic for \(E_n(x)\) with error \(O(\sqrt{n} (\log n)^{3/2})\), far better than the unconditional \(O(n \log \log n / \log^2 n)\). This also provides conditional upper bounds for the threshold \(\Psi(x)\) where \(E_n(x) > 0\).

\begin{theorem}[Conditional RH asymptotic]
Assume the Riemann Hypothesis. Then for fixed \(x \in (-2,2)\), 
\[
E_n(x) = (x+2) \frac{n}{\log n} + O(\sqrt{n} (\log n)^{3/2}).
\]
The implied constant is absolute and computable from classical RH bounds for \(\vartheta\), \(\pi\), and \emph{li} (e.g., Schoenfeld \cite{schoRH}).
\end{theorem}

\begin{proof}
Under RH, \(\vartheta(t) = t + O(\sqrt{t} \log^2 t)\) and \(\pi(t) = \li(t) + O(\sqrt{t} \log t)\) for \(t \ge t_0\). At \(t = p_n \sim n \log n\), this gives \(\vartheta(p_n) = p_n + O(\sqrt{n} (\log n)^{3/2})\). Replacing two-term PNT approximations in the unconditional proof with RH versions contributes at most \(O(\sqrt{n} (\log n)^{3/2})\) per term. Main cancellations persist, yielding the result with principal term \((x+2)n / y\).
\end{proof}

\begin{corollary}[RH positivity for \(E_n(x)\)]
Assume RH and fix \(x \in (-2,2)\). There exists an absolute computable constant \(C_{\mathrm{RH}} > 0\) such that 
\[
E_n(x) \ge (x+2) \frac{n}{y} - C_{\mathrm{RH}} \sqrt{n} \, y^{3/2}, \quad y = \log n.
\]
Thus, if \(y\) satisfies 
\[
 (x+2) \frac{n}{y} \ge  C_{\mathrm{RH}} \sqrt{n} \, y^{3/2} \quad \iff \quad y \ge 5 \log y + 2 \log \Bigl( \frac{C_{\mathrm{RH}}}{x+2} \Bigr),
\]
then \(E_n(x) > 0\). Letting \(y^\star_{\mathrm{RH}}(x)\) be the least real solution, we have 
\[
\Psi(x) \le N_x^{\mathrm{RH}} := \bigl\lceil \exp\bigl(y^\star_{\mathrm{RH}}(x)\bigr) \bigr\rceil.
\]
\end{corollary}

\begin{proof}
This follows from the theorem with absolute values and \(C_{\mathrm{RH}}\) absorbing constants. The condition implies \(E_n(x) > (x+2) n / y > 0\). The function \(f(y) = y - 5 \log y-2\log (\frac{C_{RH}}{x+2})\) is strictly increasing for \(y > 5\), ensuring uniqueness of \(y^\star_{\mathrm{RH}}(x)\).
\end{proof}

\subsection*{Evaluation of \(N_x^{\mathrm{RH}}\) in practice}

The threshold \(N_x^{\mathrm{RH}}\) solves the implicit equation 
\[
y = 5 \log y + 2\log \Bigl( \frac{C_{\mathrm{RH}}}{x+2} \Bigr).
\]
This has a unique solution \(y^\star_{\mathrm{RH}}(x)\) for \(x \in (-2,\ 2)\). It can be found by fixed-point iteration 
\[
y_{k+1} = 5 \log y_k + 2\log \Bigl( \frac{C_{\mathrm{RH}}}{x+2} \Bigr),
\]
starting from \(y_0 = 10 + \max\bigl\{0, 2 \log \bigl( \frac{C_{\mathrm{RH}}}{x+2} \bigr) \bigr\}\). Monotonicity ensures convergence.

 Schoenfeld's explicit RH bounds leads to a much tighter upper bound (for large $n$): Let \(y^\star_{\mathrm{RH}}(x)\) solve
\[
y = 5 \log y + 2 \log \left( \frac{C_{\mathrm{RH}}}{x + 2} \right),
\]
where \(C_{\mathrm{RH}} = 1/(8\pi)\) from Schoenfeld's explicit RH bounds. Then
\[
\Psi(x) \leq \lceil \exp(y^\star_{\mathrm{RH}}) \rceil.
\]
As \(x \to -2^+\), \(y^\star_{\mathrm{RH}} \sim 2 \log(1/(x+2))\), so \(\Psi(x) \sim 1/(x+2)^2\), a polynomial growth vastly smaller than the unconditional exponential.

\begin{table}[ht]
\centering
\caption{Comparison of the unconditional and conditional (RH) error terms for $E_n(x)$. The unconditional error, $E_{\text{uncond}}(n) \asymp n \log \log n / (\log n)^2$, is the dominant term for all $n > 100$. Under the Riemann Hypothesis, the error $E_{\text{RH}}(n) \asymp \sqrt{n} (\log n)^{3/2}$ becomes vanishingly small relative to the main term $(x+2)n/\log n$ and is exponentially smaller than the unconditional error for large $n$, demonstrating a huge improvement.}
\label{tab:error_comparison}
\begin{tabular}{r|r|r|r|r}
\hline
$n$ & $E_{\text{uncond}}(n)$ & $E_{\text{RH}}(n)$ & Ratio \\
    & $\approx \dfrac{n \log \log n}{(\log n)^2}$ & $\approx \sqrt{n} (\log n)^{3/2}$ & $E_{\text{uncond}} / E_{\text{RH}}$ \\
\hline
$10^{2}$  & $1.1 \times 10$  & $1.1 \times 10$  & $1.0$ \\
$10^{5}$  & $5 \times 10^{3}$  & $7.2 \times 10$  & $6.9 \times 10^{1}$ \\
$10^{10}$ & $4.2 \times 10^{8}$  & $3.1 \times 10^{3}$  & $1.4 \times 10^{5}$ \\
$10^{15}$ & $1.2 \times 10^{13}$ & $2.1 \times 10^{4}$  & $5.7 \times 10^{8}$ \\
$10^{20}$ & $4 \times 10^{17}$ & $9.6 \times 10^{4}$  & $4.2 \times 10^{12}$ \\
$10^{30}$ & $4.2 \times 10^{26}$ & $4 \times 10^{5}$  & $1.1 \times 10^{21}$ \\
$10^{50}$ & $1.2 \times 10^{45}$ & $3.9 \times 10^{6}$  & $3.1 \times 10^{38}$ \\
\hline
\end{tabular}
\end{table}

\begin{remark}
    The rigorous numerical computations and asymptotic estimates presented in this paper were carried out using \textsf{Wolfram Mathematica}~{}\cite{wolfram2023}. This includes the verification of inequalities for large initial values, the calculation of prime counting functions, and the analysis of the error term $E_n(x)$.
\end{remark}

\section{Concluding remarks and future directions}\label{conclude}

The results established in this paper give a precise asymptotic for the logarithmic
error term $E_n(x)$ in continuous Bonse-type inequalities and, in particular,
determine the exact threshold $\Psi(0.1)$. Together with the piecewise-constant
structure of $\Psi(x)$, this essentially completes the program initiated in
\cite{papa, bonse, mt1} for the specific family of exponents
\[
k(n,x)=n-\pi(n)+\frac{\pi(n)}{\pi(\log n)}-x\,\pi(\pi(n)).
\]

It is natural to ask how robust this phenomenon is. We conclude by listing a
few directions for future investigation.

\subsection*{The step structure of $\Psi(x)$}

We proved that $\Psi(x)$ is a non-increasing step function with jump set contained in the
$\alpha_n$. An interesting problem is to describe the distribution of the plateau lengths
quantitatively. For example:

\begin{question}
Do the normalized plateau lengths $(A_m-A_{m+1})\cdot m\log^2 m$ converge in distribution?
\end{question}

Understanding whether these lengths exhibit regularity or pseudorandom behavior would
shed light on the fine-scale geometry of $\Psi(x)$.

\subsection*{Asymptotic shape of $A_m$}

Our analysis shows that $A_m=-2+O(\tfrac{\log\log m}{\log m})$ as $m\to\infty$.
We conjecture a sharper expansion.

\begin{conjecture}
There exists a constant $\kappa>0$ such that
\[
A_m=-2+\kappa\,\frac{\log\log m}{\log m}+o\!\left(\frac{\log\log m}{\log m}\right).
\]
\end{conjecture}

Identifying $\kappa$ explicitly would provide the first-order correction to the limiting
domain $x>-2$ in Theorem~\ref{main1}.

\subsection*{Beyond the product of primes}
Finally, one may ask whether similar continuous inequalities hold when replacing
$p_1p_2\cdots p_n$ by other prime-related products, such as the primorial restricted
to arithmetic progressions, or partial products of twin primes or prime tuples.
Analyzing the corresponding error terms would require refined forms of the
prime number theorem in arithmetic progressions, suggesting a broad range
of possible future research.

\smallskip

We hope these conjectures and questions encourage further exploration of the
interface between Bonse-type inequalities, prime-counting functions, and the
fine distribution of prime numbers.

\addtocontents{toc}{\protect\setcounter{tocdepth}{-1}} 

\section*{Acknowledgement}
D.M. would like to acknowledge the financial support provided by the National Council for Scientific and Technological Development (CNPq).

\end{document}